\documentclass[letterpaper,10pt,twocolumn,conference]{ieeeconf}
\usepackage{amsmath,amssymb,euscript,psfrag,latexsym,graphicx}
\usepackage{bbm,color,amstext,wasysym,subfig,parskip,balance}
\usepackage{cite}
\usepackage{hyperref,breakurl}
\usepackage{amsfonts,mathrsfs,mathtools}
\usepackage{bm}
\usepackage{slashbox}
\usepackage{soul}

\newcommand*\Laplacian{\mathop{}\!\mathbin\bigtriangleup}

\newtheorem{theorem}{Theorem}

\newtheorem{lemma}{Lemma}

\newtheorem{remark}{Remark}

\newcommand{\mR}{{\mathbb R}}

\newcommand{\cG}{{\mathcal G}}
\newcommand{\cD}{{\mathscr{D}}}

\newcommand{\bbc}{\bm{c}}

\newcommand{\bbv}{\bm{v}}
\newcommand{\bbw}{\bm{w}}
\newcommand{\bbx}{{\bm{x}}}
\newcommand{\bby}{{\bm{y}}}
\newcommand{\bbz}{\bm{z}}
\newcommand{\bmu}{\bm{\mu}}
\newcommand{\btheta}{\bm{\theta}}

\newcommand{\bbmu}{\bm{\mu}}
\newcommand{\bbP}{{\bm{P}}}
\newcommand{\bbK}{\bm{K}}
\newcommand{\bbR}{\bm{R}}
\newcommand{\bbA}{\bm{A}}
\newcommand{\bbB}{\bm{B}}
\newcommand{\bbC}{\bm{C}}
\newcommand{\bbD}{\bm{D}}
\newcommand{\bbI}{\bm{I}}
\newcommand{\bbL}{\bm{L}}
\newcommand{\bbF}{\bm{F}}
\newcommand{\bbG}{\bm{G}}
\newcommand{\bbH}{\bm{H}}
\newcommand{\bbQ}{\bm{Q}}
\newcommand{\bbX}{\bm{X}}
\newcommand{\bbY}{\bm{Y}}
\newcommand{\bbZ}{\bm{Z}}
\newcommand{\bbV}{\bm{V}}

\newcommand{\bbOmega}{\bm{\Omega}}

\newcommand{\bbLambda}{\bm{\Lambda}}

\newcommand{\diag}{\operatorname{diag}}
\newcommand{\tr}{\operatorname{trace}}

\newcommand{\arginf}{\operatorname{arg\:inf}}
\newcommand{\dFR}{d_{\mathrm{FR}}}

\newcommand{\Jacobian}{\operatorname{D}}
\newcommand{\differential}{\operatorname{d}}
\renewcommand{\vec}{\operatorname{vec}}

\graphicspath{{./},{./figures/}}
\def\spacingset#1{\def\baselinestretch{#1}\small\normalsize}
\spacingset{1}

\IEEEoverridecommandlockouts
\begin{document}

\title{Gradient Flows in Filtering and Fisher-Rao Geometry}
\author{Abhishek Halder$^{1}$\thanks{$^{1}$Department of Applied Mathematics \& Statistics, University of California, Santa Cruz,
        CA 95064, USA,
        {\tt\small{ahalder@ucsc.edu}}}%
        , and Tryphon T. Georgiou$^{2}$\thanks{$^{2}$Department of Mechanical and Aerospace Engineering, University of California, Irvine,
        CA 92617, USA,
        {\tt\small{tryphon@uci.edu}}}%
        }
\maketitle

\begin{abstract}
Uncertainty propagation and filtering can be interpreted as gradient flows with respect to suitable metrics in the infinite dimensional manifold of probability density functions. Such a viewpoint has been put forth in recent literature, and a systematic way to formulate and solve the same for linear Gaussian systems has appeared in our previous work where the gradient flows were realized via proximal operators with respect to Wasserstein metric arising in optimal mass transport. In this paper, we derive the evolution equations as proximal operators with respect to Fisher-Rao metric arising in information geometry. We develop the linear Gaussian case in detail and show that a template two step optimization procedure proposed earlier by the authors still applies. Our objective is to provide new geometric interpretations of known equations in filtering, and to clarify the implication of different choices of metric.
\end{abstract}

\section{Introduction}
This paper concerns with an emerging viewpoint that gradient flows on the infinite-dimensional manifold of probability density functions can be constructed to approximate both the propagation and measurement update steps in filtering. In the systems-control literature, continuous-time filtering theory has traditionally been approached from a ``transport viewpoint" where given the noisy process and measurement models, flow of the posterior probability density function (PDF) of the state vector is described by the Kushner-Stratonovich stochastic partial differential equation (PDE) \cite{Stratonovich1960,Kushner1964} (or equivalently by the Duncan-Mortensen-Zakai PDE \cite{duncan1970likelihood,mortensen1966stochastic,zakai1969optimal} for the unnormalized posterior). In the absence of measurement update, the problem reduces to that of uncertainty propagation subject to the drift and diffusion in the process dynamics, and the evolution of the state PDF is then governed by the Fokker-Planck or Kolmogorov's forward PDE \cite{RiskenBook1989}. Due to numerical difficulties to solve these PDEs, in practice one resorts to approximate the evolution of the state PDFs using sequential Monte Carlo algorithms, which remain computationally challenging in general. This motivates us to look for alternative formulations that are theoretically equivalent to solving the associated transport PDEs.

One such alternative is the ``variational viewpoint" that has been put forth in recent literature \cite{JKO1998,LaugesenMehta2015}. Specifically, let us consider the system of It\^{o} stochastic differential equations (SDEs) for the state ($\bbx(t)$) and measurement $(\bbz(t))$ vectors
\begin{subequations}
\begin{eqnarray}
\mathrm{d}\bm{x}(t) &=& -\nabla U(\bm{x})\:\mathrm{d}t + \sqrt{2\beta^{-1}}\:\mathrm{d}\bm{w}(t),
\label{jkoSDE}\\
\mathrm{d}\bbz(t) &=& \bbc(\bm{x}(t),t)\:\mathrm{d}t + \mathrm{d}\bbv(t), 
\label{ObservationSDE}	
\end{eqnarray}
\label{CanonicalForm}
\end{subequations}
where $\bbx\in\mathbb{R}^{n}, \bbz\in\mathbb{R}^{m}, \beta>0$, $U(\cdot)$ is a potential, the process and measurement noise processes $\bbw(t)$ and $\bbv(t)$ are Wiener and satisfy $\mathbb{E}\left[\mathrm{d}w_{i}\mathrm{d}w_{j}\right] = \bbQ_{ij}\mathrm{d}t \:\forall\:i,j=1,\hdots,n$ and $\mathbb{E}\left[\mathrm{d}v_{i}\mathrm{d}v_{j}\right] = \bbR_{ij}\mathrm{d}t \:\forall\:i,j=1,\hdots,m$, with $\bbQ,\bbR \succ \bm{0}$. Also, $\bbv(t)$ is assumed to be independent of  $\bbw(t)$ and independent of the initial state $\bm{x}(0)$. In the absence of (\ref{ObservationSDE}), the Fokker-Planck PDE for the state PDF	$\rho(\bbx,t)$ corresponding to (\ref{jkoSDE}) is given by
\begin{equation}\label{eq:FP}
\frac{\partial\rho}{\partial t} = \nabla\cdot\left(\nabla U(\bm{x})\rho\right) + \beta^{-1}\Laplacian\rho, \quad \rho(\bm{x},0)=\rho_0(\bm{x}).
\end{equation} 
In the absence of (\ref{jkoSDE}), the Kushner-Stratonovich PDE \cite{Stratonovich1960,Kushner1964} for the state PDF	 $\rho^{+}(\bbx,t)$, conditioned on the history of measurements till time $t$, corresponding to (\ref{ObservationSDE}) is given by 
\begin{align}
\mathrm{d}\rho^{+}(\bbx(t),t) =& \left[ \left(\bbc(\bbx(t),t) - \mathbb{E}_{\rho^{+}}\{\bbc(\bbx(t),t)\}\right)^{\top} \bbR^{-1} \right.\nonumber\\
&\!\!\!\!\!\!\!\!\!\!\!\!\!\!\left.\left(\mathrm{d}\bbz(t) - \mathbb{E}_{\rho^{+}}\{\bbc(\bbx(t),t)\}\mathrm{d}t\right)\vphantom{\left(\bbc(\bbx(t),t) - \mathbb{E}_{\rho^{+}}\{\bbc(\bbx(t),t)\}\right)^{\top}}\right]\: \rho^{+}(\bbx(t),t).
\label{KSpde}
\end{align}
In the ``variational viewpoint" for uncertainty propagation, one sets up a discrete time-stepping scheme of the form
\begin{align}
\varrho_{k}(\bbx,h)&=\underset{\varrho}{\arginf}\; \frac{1}{2}d^2(\varrho,\varrho_{k-1}) + h\Phi(\varrho), \;\; k\in\mathbb{N},
\label{ProxPropagation}\end{align}
with step-size $h$, a distance functional $d(\cdot,\cdot)$ between two PDFs, and a functional $\Phi(\cdot)$ that depends on the drift and diffusion coefficients in the process model. The functionals $d(\cdot,\cdot)$ and $\Phi(\cdot)$ are to be chosen such that $\varrho_{k}(\bbx,h) \rightarrow \rho(\bbx,t=kh)$ as $h\downarrow 0$, thus establishing consistency between the solutions of (\ref{eq:FP}) and (\ref{ProxPropagation}). Similar recursion for the measurement update takes the form
\begin{eqnarray}
\varrho_{k}^{+}(\bbx,h)&=\underset{\varrho}{\arginf}\; \frac{1}{2}d^2(\varrho,\varrho_{k}^{-}) + h\Phi(\varrho), \;\; k\in\mathbb{N},
\label{ProxUpdate}	
\end{eqnarray}
where $\varrho_{k}^{-}$ and $\varrho_{k}^{+}$ approximates the prior and posterior PDFs, respectively, and the functional $\Phi(\cdot)$ depends on the measurement model and noisy observations. Again, the choices of $d(\cdot,\cdot)$ and $\Phi(\cdot)$ must guarantee $\varrho_{k}^{+}(\bbx,h) \rightarrow \rho^{+}(\bbx,t=kh)$ as $h\downarrow 0$, to allow consistency between the solutions of (\ref{KSpde}) and (\ref{ProxUpdate}).

Both (\ref{ProxPropagation}) and (\ref{ProxUpdate}) can be viewed as evaluating suitable (infinite dimensional) proximal operators \cite{BauschkeCombettes2011, ParikhBoyd2014} ${\rm prox}_{h\Phi}^{d}(\cdot)$ of functional $h\Phi(\cdot)$ with respect to the distance functional $d(\cdot,\cdot)$. In other words, recursion (\ref{ProxPropagation}) can be succinctly written as $\varrho_{k} = {\rm prox}_{h\Phi}^{d}(\varrho_{k-1})$. Similarly, $\varrho_{k}^{+} = {\rm prox}_{h\Phi}^{d}(\varrho_{k}^{-})$ for (\ref{ProxUpdate}). In particular, (\ref{ProxPropagation}) and (\ref{ProxUpdate}) can be seen as the gradient descent of the functional $\Phi(\cdot)$ with respect to the distance $d(\cdot,\cdot)$. For a parallel with finite dimensional gradient descent, see \cite[Section I]{HalderGeorgiouCDC2017}.

For the proximal recursion (\ref{ProxPropagation}) associated with uncertainty propagation, it was shown in \cite{JKO1998} that if $d^{2}(\cdot,\cdot)$ is chosen to be the squared Wasserstein-2 distance, which equals the cost of optimal transport between two probability measures under consideration, then $\Phi(\cdot)$ is the \emph{free energy functional}, defined to be the sum of an energy functional  and scaled negative of the differential entropy functional (see \cite[Section II]{HalderGeorgiouCDC2017} for details). In this case, the functional $\Phi(\cdot)$ depends on both $U(\cdot)$ and $\beta$ appearing in the process model (\ref{jkoSDE}).

Laugesen, Mehta, Meyn and Raginsky \cite{LaugesenMehta2015} were first to consider the proximal recursion (\ref{ProxUpdate}) associated with measurement update, where it was shown  that if $\frac{1}{2}d^{2}(\cdot,\cdot)$ is chosen to be the Kullback-Leibler divergence, then $\Phi(\cdot)$ is the ``\emph{expected quadratic surprise" functional}, given by 
\begin{align}\label{eq:Phi}
\Phi(\varrho) &:= \frac12\mathbb{E}_{\varrho}\{(\bby_{k} - \bbc(\bbx))^{\top}\bbR^{-1}(\bby_{k} - \bbc(\bbx))\},
\end{align}
where
$\bby_{k} := \frac{1}{h}\Delta \bbz_{k}$, $\Delta\bbz_{k}:=\bbz_{k} - \bbz_{k-1}$, and 
$\{\bbz_{k-1}\}_{k\in\mathbb{N}}$ is the sequence of samples of $\bbz(t)$ at $\{t_{k-1}\}_{k\in\mathbb{N}}$ for
$t_{k-1} := (k-1)h$. We refer this particular variational formulation as the LMMR scheme.

In \cite{HalderGeorgiouCDC2017}, a general framework was introduced to bring any stable controllable linear system to the canonical form (\ref{CanonicalForm}) for which the proximal recursions (\ref{ProxPropagation}) or (\ref{ProxUpdate}) apply. Also, a two step optimization strategy was proposed to compute these recursions, and using these tools, the LMMR scheme in the linear Gaussian setting was shown to recover the Kalman-Bucy filter \cite{KalmanBucy1961}. It was also found \cite[Section IV.B]{HalderGeorgiouCDC2017} that the recursion (\ref{ProxUpdate}) with $d^{2}(\cdot,\cdot)$ chosen to be the squared Wasserstein-2 distance, and $\Phi(\cdot)$ as in (\ref{eq:Phi}) in the linear Gaussian case, results a Luenberger-type estimator with static gain matrix, unlike the Kalman-Bucy filter.

It becomes apparent that the choice of the distance functional $d(\cdot,\cdot)$ cannot be arbitrary, and to appeal the gradient descent interpretation, should preferably define a metric on the manifold of PDFs. While Wasserstein-2 distance is a metric \cite[p. 208]{VillaniBook}, the Kullback-Leibler divergence is not. This naturally leads to the question whether one can find a metric $d(\cdot,\cdot)$ and functional $\Phi(\cdot)$ in (\ref{ProxUpdate}) such that the filtering equations can be seen as the gradient descent of $\Phi(\cdot)$ with respect to metric $d(\cdot,\cdot)$. In this paper, by taking $d(\cdot,\cdot)$ to be the geodesic distance induced by the Fisher-Rao metric \cite{Rao1945} and $\Phi(\cdot)$ as in (\ref{eq:Phi}), we answer this in the affirmative for the linear Gaussian setting. Specifically, a variant of the two-step optimization template proposed in \cite{HalderGeorgiouCDC2017} allows us to perform explicit computation for (\ref{ProxUpdate}) with Fisher-Rao metric, and is shown to recover the Kalman-Bucy filter in the limit $h\downarrow 0$. 

It is perhaps not surprising that the functional $\Phi(\cdot)$ in this paper and in the LMMR scheme \cite{LaugesenMehta2015} are the same, the choice of $d(\cdot,\cdot)$ are different, and yet both of these two schemes are shown to recover the same filtering equations. This is due to the well-known fact \cite[Ch. 2, p. 26--28]{Kullback1968} that the Kullback-Leibler divergence between a pair of PDFs where one PDF is infinitesimally perturbed from the other, equals to half the squared geodesic distance in Fisher-Rao metric measured between the same.

The rest of this paper is structured as follows. In Section II, we review some basic aspects of Fisher-Rao geometry on the manifold of probability density functions. In Section III, focusing on the linear Gaussian case, we approximate stochastic estimator as gradient descent with respect to the geodesic distance induced by the Fisher-Rao metric. Section IV concludes the paper.

\subsection*{Notations and Preliminaries}
As in \cite{HalderGeorgiouCDC2017}, we denote the space of PDFs on $\mR^n$ by $\cD := \{\rho: \rho \geq 0, \int_{\mR^n} \rho = 1\}$, and the space of PDFs with finite second moments by $\cD_2:=\{\rho\in\cD \mid \int_{\mR^n} \bm{x}^{\top}\bm{x}\:\rho(\bm{x})\mathrm{d}\bm{x} < \infty\}$. The notation $\cD_{\bm{\mu},\bm{P}}$ is used for the space of PDFs which share the same mean vector $\bm{\mu}$ and same covariance matrix $\bbP:=\int_{\mR^n}(\bbx-\bm{\mu})(\bbx-\bm{\mu})^\top \rho(\bbx)\mathrm{d}\bbx$. The following inclusion is immediate: $\cD_{\bm{\mu},\bm{P}} \subset \cD_{2} \subset \cD$. We use the symbol $\mathcal{N}\left(\bmu,\bbP\right)$ to denote a multivariate Gaussian PDF with mean $\bmu$, and covariance $\bm{P}$. The notation $\bbx\sim\rho$ means that the random vector $\bbx$ has PDF $\rho$; and $\mathbb{E}\left\{\cdot\right\}$ denotes the expectation operator while, when the probability density is to be specified, $\mathbb{E}_\rho\left\{\cdot\right\}:=\int_{\mR^n}(\cdot)\rho(\bbx)\mathrm{d}\bbx$.

We use $\differential(\cdot)$ for the differential, $\Jacobian(\cdot)$ for the Jacobian, and $\vec(\cdot)$ for the vectorization operator. For taking derivatives of matrix valued functions, we utilize the Jacobian identification rules \cite[p. 199, Table 2]{MagnusNeudecker}. Notation $\bbI_{n}$ stands for the $n\times n$ identity matrix, and $\parallel \cdot \parallel_{{\rm{Fro}}}$ for the Frobenius norm. The set of real $n\times n$ matrices is denoted as $\mathbb{M}_{n}$, the set of symmetric matrices as $\mathbb{S}_{n}\subset\mathbb{M}_{n}$, and the set of positive definite matrices as $\mathbb{S}_{n}^{+}\subset\mathbb{S}_{n}$. The cone $\mathbb{S}_{n}^{+}$ is a smooth differentiable manifold with tangent space $\mathcal{T}_{\bbX}\mathbb{S}_{n}^{+} = \mathbb{S}_{n}$, where $\bbX \in \mathbb{S}_{n}^{+}$. The symbols $\otimes$ and $\oplus$ denote Kronecker product and Kronecker sum\footnote{For an $m\times m$ matrix $\bbA$ and an $n\times n$ matrix $\bbB$, the Kronecker sum is $mn\times mn$ matrix $\bbA\oplus\bbB := \bbA \otimes \bbI_{n} + \bbI_{m} \otimes \bbB$.}, respectively. We will have multiple occasions to use the following three well-known facts: (1) $\vec(\cdot)$ is a linear operator, (2) $(\bbA\otimes\bbB)(\bbC\otimes\bbD) = (\bbA\bbC\otimes\bbB\bbD)$, and (3) an identity relating $\vec(\cdot)$ and Kronecker product:
\begin{eqnarray}
\vec(\bbA\bbB\bbC) = \left(\bbC^{\top}\otimes\bbA\right)\vec(\bbB).
\label{veckronprodidentity}	
\end{eqnarray}
All matrix logarithms in this paper are to be understood as principal logarithms. Furthermore, the following property (see \cite{CurtisBook}, \cite[Theorem 1.13(c)]{HighamBook}) will be useful
\begin{eqnarray}
	\bbA\log(\bbB)\bbA^{-1} = \log\left(\bbA\bbB\bbA^{-1}\right),
	\label{LogProd}
\end{eqnarray}
which also holds when $\log$ is replaced by any analytic matrix function.


\section{Fisher-Rao Geometry}\label{FRGeometrySection}
Given PDF $\rho\in\cD$ and tangent vectors $u,v\in\mathcal{T}_{\rho}\cD$, the \emph{Fisher-Rao metric} is defined as a Riemannian metric $\cG_{\rho}\left(u,v\right)$ on $\cD$, given by
\begin{eqnarray}
	\cG_{\rho}\left(u,v\right) := \displaystyle\int_{\mathbb{R}^{n}} \frac{uv}{\rho}\differential\bbx.
\label{FRMetricDefn}	
\end{eqnarray}
Let $\psi:=\sqrt{\rho}$, and consider the space of ``square-root PDFs", given by $\Psi := \{\psi: \psi \geq 0, \int_{\mathbb{R}^{n}}\psi^{2} = 1\}$. Geometrically, $\Psi$ represents the non-negative orthant of the (infinite dimensional) unit sphere imbedded in the Hilbert space of square-integrable functions on $\mathbb{R}^{n}$ equipped with the inner product $\langle \psi_{1}, \psi_{2} \rangle := \int_{\mathbb{R}^{n}}\psi_{1}\psi_{2}$. For $\psi_{i} := \sqrt{\rho_{i}}$, $i=1,2$, the (minimal) \emph{geodesic distance induced by Fisher-Rao metric}, denoted hereafter as $\dFR(\rho_{1},\rho_{2})$, is then simply the minor arc-length along the great circle connecting $\psi_{1}$ and $\psi_{2}$, i.e.,
\begin{eqnarray}
	\dFR(\rho_{1},\rho_{2}) = \arccos\langle\sqrt{\rho_{1}},\sqrt{\rho_{2}}\rangle.
	\label{GeodesicDistNonparam}
\end{eqnarray}
Since any point on the chord joining $\psi_{1}$ and $\psi_{2}$ can be represented as $(1-t)\psi_{1} + t\psi_{2}$, $t\in[0,1]$, hence the \emph{minimal geodesic} $\rho(\bbx,t)$ connecting $\rho_{1}$ and $\rho_{2}$, is obtained by parameterizing the arc $\psi(\bbx,t)$ connecting $\psi_{1}$ and $\psi_{2}$, i.e.,
\begin{align}
&\psi(\bbx,t) = \displaystyle\frac{(1-t)\psi_{1} + t\psi_{2}}{\left((1-t)^{2} + t^{2} + 2t(1-t)\langle\psi_{1},\psi_{2}\rangle\right)^{1/2}}\nonumber\\
	\Rightarrow &\rho\left(\bbx, t\right) = \displaystyle\frac{\left((1-t)\sqrt{\rho_{1}(\bbx)} + t\sqrt{\rho_{2}(\bbx)}\right)^{2}}{(1-t)^{2} + t^{2} + 2t(1-t)\langle\sqrt{\rho_{1}},\sqrt{\rho_{2}}\rangle}.
	\label{GeodesicNonparam}
\end{align}

The formula (\ref{FRMetricDefn}), (\ref{GeodesicDistNonparam}) and (\ref{GeodesicNonparam}) are non-parametric in the sense that they do not require the PDFs to have any finite dimensional parametric co-ordinate representations. 
If prior knowledge allows one to consider a known parametric family of PDFs $\rho(\bbx|\btheta)$ with $r$-dimensional parameter vector $\btheta = (\theta_{1},\hdots,\theta_{r})^{\top}$, then the Fisher-Rao metric (\ref{FRMetricDefn}) admits local coordinate representation, given by the \emph{Fisher information matrix} \cite{fisher1925theory}
\begin{align}
	g_{ij}(\btheta) &= \cG_{\rho(\bbx|\btheta)}\left(\frac{\partial}{\partial\theta_{i}}\rho(\bbx|\btheta),\frac{\partial}{\partial\theta_{j}}\rho(\bbx|\btheta)\right)\\
	&= \mathbb{E}\bigg\{\frac{\partial}{\partial\theta_{i}}\log\rho(\bbx|\btheta)\frac{\partial}{\partial\theta_{j}}\log\rho(\bbx|\btheta)\bigg\} \\
	&= -\mathbb{E}\bigg\{\frac{\partial^{2}}{\partial\theta_{i}\theta_{j}}\log\rho(\bbx|\btheta)\bigg\}, \quad i,j=1,\hdots,r.
\end{align}
The parametric version of the geodesic (\ref{GeodesicNonparam}) becomes $\rho\left(\bbx, t\right) \equiv \rho\left(\bbx | \btheta(t)\right)$, $t\in[0,1]$, where $\btheta(t)$ solves the Euler-Lagrange boundary value problem	
\begin{align}
\!\!\!\ddot{\theta}_{k}(t) \!+ \!\!\displaystyle\sum_{i,j}\Gamma_{ij}^{k}(\btheta(t))\dot{\theta}_{i}(t)\dot{\theta}_{j}(t) = 0, \; i,j,k=1,\hdots,r,
\label{EulerLagrangeParametric}
\end{align}
with $\btheta_{1}:=\btheta(0)$ and $\btheta_{2}:=\btheta(1)$ given, and $\Gamma_{ij}^{k}$ are Christoffel symbols of the second kind defined via $g_{ij}$. The squared arc-length 
\begin{eqnarray}
(\differential s)^{2} = \sum_{i,j}g_{ij}(\btheta(t))\differential\theta_{i}(t)\differential\theta_{j}(t),	
\label{dsSquareParametric}
\end{eqnarray}
and a parametric version of the geodesic distance (\ref{GeodesicDistNonparam}) is 
{\small{\begin{align}
	\dFR\left(\rho(\bbx|\btheta_{1}),\rho(\bbx|\btheta_{2})\right) &= \!\int_{t=0}^{t=1}\!\!\!\differential s(t) \nonumber\\
	&= \!\displaystyle\int_{0}^{1}\!\!\!\sqrt{\sum_{i,j}g_{ij}(\btheta(t))\dot{\theta}_{i}(t)\dot{\theta}_{j}(t)}\:\differential t. 
\end{align}}}
Of particular interest to us, is the family of multivariate Gaussian PDFs.
In this case, for $\bmu,\bmu_{1},\bmu_{2}\in\mathbb{R}^{n}$, and $\bbP,\bbP_{1},\bbP_{2} \in \mathbb{S}_{n}^{+}$, (\ref{EulerLagrangeParametric}) can be written as \cite[Theorem 6.1]{Skovgaard1981}
\begin{subequations}
	\begin{align}
\ddot{\bmu} - \dot{\bbP}\bbP^{-1}\dot{\bmu} = \bm{0},\\
\ddot{\bbP} + \dot{\bmu}\dot{\bmu}^{\top} - 	\dot{\bbP}\bbP^{-1}\dot{\bbP} = \bm{0},
\end{align}
\end{subequations}
with $t\in[0,1]$, $(\bbmu_{1},\bbP_{1}) := (\bbmu(0),\bbP(0))$ and $(\bbmu_{2},\bbP_{2}) := (\bbmu(1),\bbP(1))$. Furthermore, (\ref{dsSquareParametric}) becomes
\begin{align}
	(\differential s)^{2} = \differential\bmu^{\top}\bbP^{-1}\differential\bbmu + \frac{1}{2}\tr((\bbP\differential\bbP)^{2}).
\end{align}
The following special cases are well-known \cite{AtkinsonMitchell1981,Skovgaard1981,rao1987differential}:
{\small{\begin{align}
	\dFR(\mathcal{N}\!\!\left(\bmu_{1},\bbP\right),\mathcal{N}\!\!\left(\bmu_{2},\bbP\right)) = \left(\!(\bmu_{1}-\bmu_{2})^{\!\top}\!\bbP^{-1}(\bmu_{1}-\bmu_{2})\!\!\right)^{\!1/2},
	\label{dFRmu1mu2}
\end{align}}}
\!\!which is also the Mahalanobis distance \cite{Mahalanobis1936}, and that
{\small{\begin{align}
	\dFR(\bbP_{1},\bbP_{2}) &:= \dFR(\mathcal{N}\!\!\left(\bmu,\bbP_{1}\right),\mathcal{N}\!\!\left(\bmu,\bbP_{2}\right)) \label{dFRP1P2}\\
	&= \left\Vert \frac{1}{\sqrt{2}}\log\left(\bbP_{1}^{-1/2}\bbP_{2}\bbP_{1}^{-1/2}\right) \right\Vert_{\rm{Fro}}\nonumber\\
	&= \bigg\{\!\frac{1}{2}\tr\!\left[\!\left(\log\left(\bbP_{1}^{-1/2}\bbP_{2}\bbP_{1}^{-1/2}\right)\right)^{2}\right]\!\!\bigg\}^{\!1/2}\label{dFRbetweenCovMatrices}\\
	&= \left(\frac{1}{2}\displaystyle\sum_{i=1}^{n}\left(\log\lambda_{i}\right)^{2}\right)^{\!1/2},\nonumber
\end{align}}}
\!\!where $\lambda_{1},\hdots,\lambda_{n}$ are the eigenvalues of $\bbP_{1}^{-1}\bbP_{2}$. To the best of our knowledge, no closed-form expression is known for $\dFR(\mathcal{N}(\bmu_{1},\bbP_{1}),\mathcal{N}(\bmu_{2},\bbP_{2}))$.


\section{Proximal Recursion with respect to Fisher-Rao Metric}\label{MainSection}
In this Section, we consider approximating continuous-time stochastic estimator via the recursive variational scheme
{\small{\begin{eqnarray}
\varrho_{k}^{+}(\bbx,h)= {\rm prox}_{h\Phi}^{\dFR}(\varrho_{k}^{-}) = \underset{\varrho\in\cD_{2}}{\arginf}\; \frac{1}{2}\dFR^2(\varrho,\varrho_{k}^{-}) + h\Phi(\varrho),
\label{ProxFR}	
\end{eqnarray}}}
\!\!\!where $k\in\mathbb{N}$, $h$ is the step-size, and $\varrho_{k}^{-}$ and $\varrho_{k}^{+}$ are the approximations of the prior and posterior PDFs, respectively. We choose the functional $\Phi(\cdot)$ as in (\ref{eq:Phi}), and focus on the linear Gaussian case with process and measurement models
\begin{subequations}
	\begin{align}
\mathrm{d}\bbx(t) &= \bbA\bbx(t)\mathrm{d}t + \sqrt{2}\bbB\mathrm{d}\bbw(t),\label{eq:LinDyn}\\
\mathrm{d}\bbz(t) &= \bbC\bm{x}(t)\:\mathrm{d}t + \mathrm{d}\bbv(t),\label{eq:LinMeas}
\end{align}
\end{subequations}
where $\bbC\in\mathbb{R}^{m\times n}$ and $\rho_{0} = \mathcal{N}(\bbmu_{0},\bbP_{0})$. It is well-known that $\rho^{+}(\bbx(t),t) = \mathcal{N}(\bbmu^{+}(t),\bbP^{+}(t))$ and the optimal estimator is given by the \emph{Kalman-Bucy filter} \cite{KalmanBucy1961}, consisting of the following SDE-ODE system:

\vspace*{-10pt}
\begin{subequations}
{\small{\begin{align}
&\mathrm{d}\bbmu^{+}(t) = \bbA\bbmu^{+}(t)\mathrm{d}t + \bbK(t)\left(\mathrm{d}\bbz(t) - \bbC\bbmu^{+}(t)\mathrm{d}t\right), \label{eq:KB1}\\
& \! \!\dot{\bbP}^{+} \!(t)  \!=  \!\bbA\bbP^{+} \!(t) \!+  \!\! \bbP^{+} \!(t)\bbA^{\top}  \! \!+ \! 2\bbB\bbB^{\top}    \!\!\!-  \!\bbK \!(t)\bbR\bbK \!(t)^{\top}, \label{eq:KB2}
\end{align}}}
\label{KB}
\end{subequations}
where $\bbK(t) := \bbP^{+}(t)\bbC^{\top}\bbR^{-1}$ is the Kalman gain.

We next demonstrate that evaluating the proximal recursion (\ref{ProxFR}) for the linear Gaussian case in the $h\downarrow 0$ limit, indeed recovers the Kalman-Bucy filter (\ref{KB}). To carry out the optimization over $\cD_{2}$ in (\ref{ProxFR}), we adopt a two step strategy proposed in \cite{HalderGeorgiouCDC2017}. In the \textbf{first step}, we optimize the objective in (\ref{ProxFR}) over the parameterized subspace $\cD_{\bmu,\bbP} \subset \cD_{2}$. In the \textbf{second step}, we optimize over the subspace parameters $\bmu$ and $\bbP$. The main insight behind this strategy comes from the fact that the objective in (\ref{ProxFR}) is a sum of two functionals. If we can find a parameterized subspace of $\cD_{2}$, over which the individual $\arginf$s of these two functionals match, then it must also be the $\arginf$ of their sum. We detail these steps below. The ensuing calculations require some technical results collected in the Appendix.

\subsection{Two Step Optimization}

{\em 1) Optimizing over $\cD_{\bmu,\bbP}$:}
Let $\varrho_{k}^{-}=\mathcal{N}(\bmu_{k}^{-},\bbP_{k}^{-})$ be the prior state PDF at time $t=kh$. From Lemma \ref{MinFisherRaoFromGaussianUnderCovConstraint} in the Appendix, we know
\begin{eqnarray}
\underset{\varrho\in\cD_{\bmu,\bbP}}{\arginf} \: \dFR^{2}\left(\varrho,\mathcal{N}(\bmu_{k}^{-},\bbP_{k}^{-})\right) = \mathcal{N}(\bmu,\bbP).
\label{minFRoverDmuP}	
\end{eqnarray}
On the other hand, for any $\varrho\in\cD_{\bmu,\bbP}$, we have
{\small{\begin{align}
&\frac{1}{2}\:\mathbb{E}_{\varrho}\{(\bby_{k} - \bbC\bbx)^{\top} \bbR^{-1} (\bby_{k} - \bbC\bbx)\} = \frac{1}{2}\vphantom{\tr\left(\bbC^{\top}\bbR^{-1}\bbC\bbP\right)}\left[(\bby_{k} - \bbC\bbmu)^{\top} \right.\nonumber\\
&\left.\bbR^{-1}(\bby_{k} - \bbC\bbmu)+\tr\left(\bbC^{\top}\bbR^{-1}\bbC\bbP\right)\right] = \text{constant}.
\label{ConstPhi}	
\end{align}}}
Consequently
{\small{\begin{align}
	&\underset{\varrho\in\cD_{\bmu,\bbP}}{\arginf} \: \vphantom{}\left[\frac{1}{2}\dFR^{2}\left(\varrho,\mathcal{N}(\bmu_{k}^{-},\bbP_{k}^{-})\right) + \frac{h}{2}\:\mathbb{E}_{\varrho}\{(\bby_{k}  - \bbC\bbx)^{\top} \bbR^{-1} \right.\nonumber\\
	& \hspace*{3.9cm} \left.(\bby_{k} - \bbC\bbx)\}\vphantom{\frac{h}{2}\:\mathbb{E}_{\varrho}\{(\bby_{k}  - \bbC\bbx)^{\top} \bbR^{-1}}\right] = \mathcal{N}(\bbmu,\bbP),
\label{arginfcombined}
\end{align}}}
and the corresponding infimum value is
{\small{\begin{align*}
&\frac{1}{2}\dFR^{2}(\mathcal{N}(\bmu,\bbP),\mathcal{N}(\bmu_{k}^{-},\bbP_{k}^{-})) + \frac{h}{2}\left[(\bby_{k} - \bbC\bbmu)^{\top}\bbR^{-1}\right.\nonumber\\
& \left.\hspace*{3cm}(\bby_{k} - \bbC\bbmu)+\tr\left(\bbC^{\top}\bbR^{-1}\bbC\bbP\right)\right],	
\end{align*}}}
\!\!\!\!\!\!which is not convenient for our next step due to the lack of availability of closed-form expression for $\dFR^{2}(\mathcal{N}(\bmu,\bbP),\mathcal{N}(\bmu_{k}^{-},\bbP_{k}^{-}))$. This can be circumvented as follows. Instead of choosing $\cD_{\bmu,\bbP}$ with two free parameters, we choose two different parameterized subspaces, viz. $\cD_{\bmu,\bbP_{k}^{-}}$ and $\cD_{\bmu_{k}^{-},\bbP}$, each with one free parameter. We next carry out the two step optimization strategy for each of the two one-parameter subspaces.

Similar to (\ref{arginfcombined}), the $\arginf$ for $\cD_{\bmu,\bbP_{k}^{-}}$ is $\mathcal{N}(\bmu,\bbP_{k}^{-})$, and using (\ref{dFRmu1mu2}), the corresponding infimum is
{\small{\begin{align}
&\frac{1}{2}(\bmu-\bmu_{k}^{-})^{\!\top}\!\bbP_{k}^{-1}(\bmu-\bmu_{k}^{-}) + \frac{h}{2}\left[(\bby_{k} - \bbC\bbmu)^{\top}\bbR^{-1}(\bby_{k} - \bbC\bbmu)\right.\nonumber\\
& \left.\hspace*{3.4cm}+\tr\left(\bbC^{\top}\bbR^{-1}\bbC\bbP_{k}^{-}\right)\right].
\label{infimummufree}	
\end{align}}}

Likewise, the $\arginf$ for $\cD_{\bmu_{k}^{-},\bbP}$ is $\mathcal{N}(\bmu_{k}^{-},\bbP)$, and using the notation (\ref{dFRP1P2}), the corresponding infimum is
{\small{\begin{align}
&\frac{1}{2}\dFR^{2}(\bbP,\bbP_{k}^{-}) + \frac{h}{2}\left[(\bby_{k} - \bbC\bbmu_{k}^{-})^{\top}\bbR^{-1}(\bby_{k} - \bbC\bbmu_{k}^{-})\right.\nonumber\\
& \left.\hspace*{4cm}+\tr\left(\bbC^{\top}\bbR^{-1}\bbC\bbP\right)\right].
\label{infimumPfree}	
\end{align}}}

{\em 2) Optimizing over $(\bmu,\bbP)$:}
Equating the partial derivative of (\ref{infimummufree}) w.r.t. $\bbmu$ to zero, and setting $\bmu \equiv \bmu_{k}^{+}$ in the resulting equation, we obtain
\begin{align}
&(\bbP_{k}^{-})^{-1}\left(\bbmu_{k}^{-} - \bbmu_{k}^{+}\right) + h\bbC^{\top}\bbR^{-1}\left(\bby_{k} - \bbC\bbmu_{k}^{+}\right) = \bm{0}, \nonumber\\
\Rightarrow &\bbmu_{k}^{+} = \bbmu_{k}^{-} + h\bbP_{k}^{-}\bbC^{\top}\bbR^{-1}\left(\bby_{k}-\bbC\bbmu_{k}^{+}\right).
\label{IntermediateMeanSDE}	
\end{align}

On the other hand, equating the partial derivative of (\ref{infimumPfree}) w.r.t. $\bbP$ to zero, using Theorem \ref{PartialofFisherRaoGeodesicLength} in Appendix, and then setting $\bbP \equiv \bbP_{k}^{+}$ in the resulting algebraic equation, we get
\begin{align}
	&\frac{1}{2}\left(\bbP_{k}^{+}\right)^{-1}\log\left(\bbP_{k}^{+}\left(\bbP_{k}^{-}\right)^{-1}\right) + \frac{h}{2} \bbC^{\top}\bbR^{-1}\bbC = 0,\nonumber\\
\Rightarrow &\bbP_{k}^{+}\left(\bbP_{k}^{-}\right)^{-1} = \exp\left(-h\bbP_{k}^{+}\bbC^{\top}\bbR^{-1}\bbC\right) \nonumber\\
&\qquad\qquad\quad\;\: = \bbI_{n} -h\bbP_{k}^{+}\bbC^{\top}\bbR^{-1}\bbC + O(h^2).	
\label{PfreeStep2Intmdt1}
\end{align}
Pre-multiplying both sides of (\ref{PfreeStep2Intmdt1}) by $\left(\bbP_{k}^{+}\right)^{-1}$ results
\begin{align}
\left(\bbP_{k}^{+}\right)^{-1} &= \left(\bbP_{k}^{-}\right)^{-1} \left(\bbI_{n} + h\bbP_{k}^{-}\bbC^{\top}\bbR^{-1}\bbC\right) +  O(h^2),\nonumber\\
\Rightarrow \bbP_{k}^{+} &= \left(\bbI_{n} + h\bbP_{k}^{-}\bbC^{\top}\bbR^{-1}\bbC\right)^{-1}\bbP_{k}^{-} + O(h^2)\nonumber\\
& = \bbP_{k}^{-} - h\bbP_{k}^{-}\bbC^{\top}\bbR^{-1}\bbC\bbP_{k}^{-} + O(h^2).
\label{PfreeStep2Intmdt2}	
\end{align}
Interestingly, equations (\ref{IntermediateMeanSDE}) and (\ref{PfreeStep2Intmdt2}) are same as those obtained (see equations (35) and (36) in \cite{HalderGeorgiouCDC2017}) by optimizing (\ref{ProxUpdate}) over $\cD_{\bmu,\bbP}$ with same $\Phi(\cdot)$ as in (\ref{eq:Phi}), but $\frac{1}{2}d^{2}(\cdot,\cdot)$ chosen to be the Kullback-Leibler divergence.

\subsection{Recovering the Kalman-Bucy Filter}
We recall two basic relations (equations (29) and (30) in \cite{HalderGeorgiouCDC2017}) for the propagation step obtained by discrete-time stepping of (\ref{ProxPropagation}) with Wasserstein-2 distance as $d(\cdot,\cdot)$, and free-energy functional as $\Phi(\cdot)$, for the linear Gaussian process model (\ref{eq:LinDyn}):
\begin{align}
	\!\!\!\!\bmu_{k} &=\! \left(\bbI_{n} + h\bbA\right)\bmu_{k-1} + O(h^2), \label{muPrior}\\
	\!\!\!\!\bbP_{k} &= \!\bbP_{k-1} \!+ h\left(\bbA\bbP_{k-1} \!+\! \bbP_{k-1}\bbA^{\!\top}\! \!+\! 2\bbB\bbB^{\!\top}\right)\! + \!O(h^2). \label{PPrior}
\end{align}
We refer the readers to \cite[Section III.B]{HalderGeorgiouCDC2017} for details of their derivations. Intuitively, (\ref{muPrior}) and (\ref{PPrior}) can be viewed as the Euler discretizations of the well-known mean and covariance (Lyapunov) ODEs for uncertainty propagation.

Using $\Delta \bbz_{k} = \bby_{k}h$ (from Section I), $\differential\bbz(t) = \Delta \bbz_{k} + O(h^2)$, $\bmu^{+}(t)\differential t = \bmu_{k}^{+}h + O(h^2)$, (\ref{muPrior}) and (\ref{PPrior}), we simplify (\ref{IntermediateMeanSDE}) as
{\small{\begin{eqnarray*}
\bbmu_{k}^{+} \!\!-\! \bbmu_{k-1}^{+} = h\bbA\bbmu_{k-1}^{+} + \bbP_{k}^{+}\bbC^{\top}\bbR^{-1}\left(\Delta\bbz_{k} - h\bbC\bbmu_{k}^{+}\right) + O(h^{2}),
\end{eqnarray*}}}
which in the limit $h\downarrow 0$, recovers (\ref{eq:KB1}).

On the other hand, combining (\ref{PfreeStep2Intmdt2}) with (\ref{PPrior}) results
\begin{eqnarray*}
\bbP_{k}^{+} - \bbP_{k-1}^{+} = h(\bbA\bbP_{k-1}^{+} + \bbP_{k-1}^{+}\bbA^{\top} + 2\bbB\bbB^{\top}) \nonumber\\
- h\bbP_{k-1}^{+}\bbC^{\top}\bbR^{-1}\bbC\bbP_{k-1}^{+} \:+\: O(h^{2}),	
\end{eqnarray*}
which in the limit $h\downarrow 0$, recovers (\ref{eq:KB2}).

Thus we have demonstrated that the measurement update step in Kalman-Bucy filter (\ref{KB}) can be interpreted as the gradient descent of functional (\ref{eq:Phi}) with respect to the metric $\dFR$ for $\bbc(\bbx)=\bbC\bbx$. This complements our earlier result \cite[Section III]{HalderGeorgiouCDC2017} showing the propagation step can be interpreted as the gradient descent of certain free energy functional with respect to the Wasserstein-2 metric.

%

\section{Conclusions} 
This paper contributes to an emerging research program that views the filtering equations as gradient flux or steepest descent on the manifold of probability density functions. For gradient descent to be meaningful in infinite dimensions, one requires a metric with respect to which distance between probability density functions are to be measured. By using the geodesic distance induced by the Fisher-Rao metric, we have shown equivalence between a discrete time-stepping procedure for gradient descent on the manifold of conditional probability density functions and the filtering equations in the linear Gaussian setting. Here, our intent has been to understand the underlying geometric implications, and to work out the theoretical details to recover known facts. Our hope is that the results in this paper will help motivate developing these ideas in nonlinear filtering setting and to solve them via proximal algorithms \cite{ParikhBoyd2014}.


\appendix
In this Appendix, we collect several technical results that are used in Section \ref{MainSection}.

\begin{lemma}\label{MinFisherRaoFromGaussianUnderCovConstraint}
	Given $\bmu,\bmu_{0}\in\mathbb{R}^{n}$ and $\bbP,\bbP_{0} \in \mathbb{S}_{n}^{+}$, 
	\begin{align*}
		\underset{\rho\in\cD_{\bmu,\bbP}}{\arginf} \; \dFR^{2}\left(\rho,\mathcal{N}\left(\bmu_{0},\bbP_{0}\right)\right)= \mathcal{N}\left(\bmu,\bbP\right).
	\end{align*}
\end{lemma}
\begin{proof}
This is a consequence of Cram\'er-Rao inequality; see Theorem 20 in \cite[p. 1512]{DemboCoverThomas1991}, also Section 1 in \cite{Bercher2012}.  	
\end{proof}

Lemma \ref{ThreeMatrixFuncJacobianLemma} and \ref{MatrixIntegralOfInverseLoginverse} below provide preparatory steps for proving Theorem \ref{PartialofFisherRaoGeodesicLength} that follows.

\begin{lemma}\label{ThreeMatrixFuncJacobianLemma}
The Jacobians of the matrix valued functions $\bbF(\bbX):=\bbX^{2}$ for $\bbX \in \mathbb{M}_{n}$, $\bbG(\bbX):=\log\bbX$ for $\bbX\in\mathbb{S}_{n}^{+}$, and $\bbH(\bbX):=\bbA\bbX\bbB$ for real $\bbA,\bbX,\bbB$ of sizes such that the product is defined, are respectively given by\\
(i) $\Jacobian\bbF(\bbX) = \bbX^{\top} \oplus \bbX$,\\
(ii) $\Jacobian\bbG(\bbX) = \int_{0}^{\infty}((\bbX + \tau\bbI_{n}) \otimes (\bbX + \tau\bbI_{n}))^{-1}\:\differential\tau$,\\
(iii) $\Jacobian\bbH(\bbX) = \bbB^{\top} \otimes \bbA.$ 	
\end{lemma}
\begin{proof}
(i) Taking the differential operator $\differential(\cdot)$ to both sides of $\bbF(\bbX) = \bbX\bbX$ results $\differential\bbF(\bbX) = (\differential\bbX)\bbX + \bbX(\differential\bbX)$. Applying $\vec(\cdot)$ to both sides of this resulting expression, using (\ref{veckronprodidentity}), and noticing that $\differential\vec(\bbF(\bbX))= \Jacobian\bbF(\bbX)\differential\vec(\bbX)$, we arrive at $\Jacobian\bbF(\bbX) = \bbX^{\top} \oplus \bbX$ as claimed. 
\\
(ii) The directional derivative of $\bbG(\bbX)$ in the direction $\bbZ\in\mathcal{T}_{\bbX}\mathbb{S}_{n}^{+} = \mathbb{S}_{n}$, denoted as $\differential\bbG_{\bbZ}(\bbX)$, is given by (see \cite[Section 3.1]{Ruskai2005}, also \cite[Appendix I.B]{TryphonTIT2006})
\begin{align}
\differential\bbG_{\bbZ}(\bbX) &:= \displaystyle\lim_{h\rightarrow 0} \displaystyle\frac{\log\left(\bbX + h\bbZ\right) - \log\bbX}{h}\nonumber\\
&= \displaystyle\int_{0}^{\infty}(\bbX + \tau\bbI_{n})^{-1}\bbZ(\bbX + \tau\bbI_{n})^{-1}\differential\tau.  
\label{DirectionalDerivativeOfGXindirectionZ}	
\end{align}
Recall that
\begin{eqnarray}
\vec\left(\differential\bbG_{\bbZ}(\bbX)\right) = \Jacobian\bbG(\bbX)\vec(\bbZ) \label{DirectionalDerivative2Jacobian}.	
\end{eqnarray}
Applying $\vec(\cdot)$ to both sides of (\ref{DirectionalDerivativeOfGXindirectionZ}) yields
\begin{align}
&\vec\left(\differential\bbG_{\bbZ}(\bbX)\right) \!= \!\displaystyle\int_{0}^{\infty}\!\!\!\!\!\! \vec\left((\bbX + \tau\bbI_{n})^{-1}\bbZ(\bbX + \tau\bbI_{n})^{-1}\right) \differential\tau 	\nonumber\\
&\stackrel{\footnotesize{(\ref{veckronprodidentity})}}{=}  \int_{0}^{\infty}\!\!\left((\bbX + \tau\bbI_{n})^{-\top} \otimes (\bbX + \tau\bbI_{n})^{-1}\right)\vec(\bbZ)\:\differential\tau \nonumber\\
&= \left(\int_{0}^{\infty}\!\!((\bbX + \tau\bbI_{n}) \otimes (\bbX + \tau\bbI_{n}))^{-1}\:\differential\tau\right)\vec(\bbZ),
\label{vecAndIntegral}
\end{align} 
where the last step is due to the fact that inverse of Kronecker product equals Kronecker product of inverses (in the same order), and that $\bbX$ is symmetric. Equating (\ref{DirectionalDerivative2Jacobian}) and (\ref{vecAndIntegral}), the statement follows.
\\
(iii) See \cite[Appendix A, Lemma 3]{HalderWendelACC2016}.	
\end{proof}

\begin{lemma}\label{MatrixIntegralOfInverseLoginverse}
For $\bbX\in\mathbb{S}_{n}^{+}$, $\displaystyle\int_{0}^{\infty}\!\!\!\!(\bbX + \tau\bbI_{n})^{-1}\left(\log\bbX\right)(\bbX + \tau\bbI_{n})^{-1}\differential\tau \allowbreak = \bbX^{-1}\log\bbX =\allowbreak (\log\bbX)\bbX^{-1}$.		
\end{lemma}
\begin{proof}
Consider the spectral decomposition $\bbX = \bbV \bbLambda \bbV^{-1}$, where $\bbLambda:=\diag(\lambda_{i})$ is positive diagonal matrix. Then $\log\bbX = \bbV \left(\log\bbLambda\right) \bbV^{-1}$. Writing $\bbI_{n} = \bbV\bbV^{-1}$, the integral of interest equals 
\begin{align*}
	\bbV\diag\left(\int_{0}^{\infty}\!\!\!\!\frac{\log\lambda_{i}}{(\lambda_{i}+\tau)^{2}}\differential\tau\right)\bbV^{-1} = \bbV\diag\left(\frac{\log\lambda_{i}}{\lambda_{i}}\right)\bbV^{-1}.
\end{align*}
Hence the result.
\end{proof}

\begin{theorem}\label{PartialofFisherRaoGeodesicLength}
	For $\bbP,\bbP_{0} \in \mathbb{S}_{n}^{+}$,
	\begin{align*}
		\displaystyle\frac{\partial}{\partial\bbP}\dFR^{2}(\bbP,\bbP_{0}) = \bbP^{-1}\log\left(\bbP\bbP_{0}^{-1}\right)= \log\left(\bbP_{0}^{-1}\bbP\right)\bbP^{-1}.
	\end{align*}
\end{theorem}
\begin{proof}
Let $\bbF(\bbX) := \bbX^{2}$, $\bbG(\bbX) := \log\bbX$, $\bbH(\bbX) := \bbP_{0}^{-1/2}\bbX\bbP_{0}^{-1/2}$, and $\bbOmega := \frac{1}{2}\bbF \circ \bbG \circ \bbH$. From (\ref{dFRbetweenCovMatrices}), it is clear that our objective is to compute $\bbY := \frac{\partial}{\partial\bbP}\tr\left(\bbOmega(\bbP)\right)$.

Notice that
{\small{\begin{align}
&\tr\left(\differential\bbOmega(\bbP)\right) \nonumber\\
&= \left(\vec(\bbI_{n})\right)^{\!\top\!\!} \vec\left(\differential\bbOmega(\bbP)\right) \nonumber\\
&= \left(\vec(\bbI_{n})\right)^{\!\top\!\!} \Jacobian\bbOmega(\bbP) \differential\vec(\bbP) \nonumber\\
&= \frac{1}{2}\left(\vec(\bbI_{n})\right)^{\!\top\!\!} \Jacobian\bbF\left(\bbG\left(\bbH(\bbP)\right)\right) \Jacobian\bbG\left(\bbH(\bbP)\right) \Jacobian\bbH(\bbP) \nonumber\\
&\qquad\qquad\qquad\qquad\qquad\qquad\qquad\qquad\qquad\qquad\differential\vec(\bbP)	\label{ChainRuleJac}
\end{align}}}
\noindent where the last step is the chain rule for Jacobians. Since
\begin{eqnarray*}
\bbP,\allowbreak\bbP_{0},\underbrace{\bbP_{0}^{-1/2}\bbP\bbP_{0}^{-1/2}}_{=\bbH(\bbP)} \in \mathbb{S}_{n}^{+},	\underbrace{\log\left(\bbP_{0}^{-1/2}\bbP\bbP_{0}^{-1/2}\right)}_{=\bbG(\bbH(\bbP))} \in \mathbb{S}_{n},
\end{eqnarray*}
therefore using Lemma \ref{ThreeMatrixFuncJacobianLemma}, we have 
\begin{align}
&\Jacobian\bbH(\bbP) = \bbP_{0}^{-1/2} \otimes \bbP_{0}^{-1/2} \in \mathbb{S}_{n}^{+}, \label{Jac1}\\
&\Jacobian\bbG\left(\bbH(\bbP)\right) = \displaystyle\int_{0}^{\infty}\!\!\left((\bbP_{0}^{-1/2}\bbP\bbP_{0}^{-1/2} + \tau\bbI_{n}) \:\otimes \right. \nonumber\\
&\qquad\qquad\qquad\left.(\bbP_{0}^{-1/2}\bbP\bbP_{0}^{-1/2} + \tau\bbI_{n})\right)^{-1}\differential\tau \in \mathbb{S}_{n}^{+}, \label{Jac2}\\
& \Jacobian\bbF\left(\bbG\left(\bbH(\bbP)\right)\right) = \log\left(\bbP_{0}^{-1/2}\bbP\bbP_{0}^{-1/2}\right) \oplus \nonumber\\
&\qquad\qquad\qquad\qquad\qquad\log\left(\bbP_{0}^{-1/2}\bbP\bbP_{0}^{-1/2}\right) \in \mathbb{S}_{n}. \label{Jac3}	
\end{align}
By virtue of the definition $\bbY := \frac{\partial}{\partial\bbP}\tr\left(\bbOmega(\bbP)\right)$, we also have
\begin{eqnarray}
\differential\tr(\bbOmega(\bbP)) = \vec\left(\bbY^{\top}\right)	\differential\vec(\bbP).
\label{Partial2differential}
\end{eqnarray}
Comparing the right hand sides of (\ref{ChainRuleJac}) and (\ref{Partial2differential}), substituting (\ref{Jac1}), (\ref{Jac2}), (\ref{Jac3}), and using the shorthand $\bbH$ for $\bbH(\bbP) = \bbP_{0}^{-1/2}\bbP\bbP_{0}^{-1/2}$, we obtain
\begin{align}
&\vec(\bbY) = \frac{1}{2}\left(\bbP_{0}^{-1/2} \otimes \bbP_{0}^{-1/2}\right) \left(\displaystyle\int_{0}^{\infty}\!\!\left((\bbH + \tau\bbI_{n}) \:\otimes\: \right.\right. \nonumber\\
&\qquad\qquad\left.\left.(\bbH + \tau\bbI_{n})\right)^{-1}\differential\tau\right)\left(\log\bbH \oplus \log\bbH\right)\vec(\bbI_{n}) \nonumber\\
&\Rightarrow \left(\bbP_{0}^{1/2} \otimes \bbP_{0}^{1/2}\right)\vec(\bbY) = \frac{1}{2}\left(\displaystyle\int_{0}^{\infty}\!\!\left((\bbH + \tau\bbI_{n}) \:\otimes\: \right.\right. \nonumber\\
&\qquad\left.\left.(\bbH + \tau\bbI_{n})\right)^{-1}\differential\tau\right)\left(\log\bbH \oplus \log\bbH\right)\vec(\bbI_{n}).
\label{vecY}	
\end{align}
Using (\ref{veckronprodidentity}), we rewrite the LHS of (\ref{vecY}) as $\vec\left(\bbP_{0}^{1/2}\bbY\bbP_{0}^{1/2}\right)$. On the other hand, observe that $\left(\log\bbH \oplus \log\bbH\right)\allowbreak\vec(\bbI_{n}) = \left(\log\bbH\allowbreak\otimes\bbI_{n} \allowbreak + \bbI_{n}\otimes\log\bbH\right)\vec(\bbI_{n})=\vec\left(2\log\bbH\right)$ (using (\ref{veckronprodidentity}) again). Thus (\ref{vecY}) simplifies to 
\begin{align}
&\vec\left(\bbP_{0}^{1/2}\bbY\bbP_{0}^{1/2}\right) \nonumber\\
&= \left(\displaystyle\int_{0}^{\infty}\!\!\!\!(\bbH + \tau\bbI_{n})^{-1}\otimes(\bbH + \tau\bbI_{n})^{-1}\differential\tau\right)\vec(\log\bbH) \label{vecYFirstStep}\\
&=  \displaystyle\int_{0}^{\infty}\!\!\!\!\left((\bbH + \tau\bbI_{n})^{-1}\otimes(\bbH + \tau\bbI_{n})^{-1}\right) \vec(\log\bbH)\differential\tau \nonumber\\
&\stackrel{\small{(2)}}{=} \displaystyle\int_{0}^{\infty}\!\!\!\!\vec\left((\bbH + \tau\bbI_{n})^{-1}\left(\log\bbH\right)(\bbH + \tau\bbI_{n})^{-1}\right)\differential\tau \nonumber\\
&= \vec\left(\displaystyle\int_{0}^{\infty}\!\!\!\!(\bbH + \tau\bbI_{n})^{-1}\left(\log\bbH\right)(\bbH + \tau\bbI_{n})^{-1}\differential\tau\right)\nonumber\\
&\stackrel{{\footnotesize{(\text{Lemma} \:\ref{MatrixIntegralOfInverseLoginverse})}}}{=} \!\vec\left(\bbH^{-1}\log\bbH\right)\! = \!\vec\left(\left(\log\bbH\right)\bbH^{-1}\right),
\label{Simplified}	
\end{align}
which yields the matrix equation 
\begin{align}
&\bbP_{0}^{1/2}\bbY\bbP_{0}^{1/2} = \bbH^{-1}\log\bbH = \left(\log\bbH\right)\bbH^{-1}\nonumber\\
\Rightarrow &\bbY = \bbP^{-1}\bbP_{0}^{1/2}\log\left(\bbP_{0}^{-1/2}\bbP\bbP_{0}^{-1/2}\right)\bbP_{0}^{-1/2}, \nonumber\\
&\quad = \bbP_{0}^{-1/2}\log\left(\bbP_{0}^{-1/2}\bbP\bbP_{0}^{-1/2}\right)\bbP_{0}^{1/2}\bbP^{-1}.
\label{FinalY}
\end{align}
Using (\ref{LogProd}) on (\ref{FinalY}), the proof is complete.
\end{proof}

\begin{remark}
Starting from (\ref{vecYFirstStep}), one can alternatively derive (\ref{FinalY}) via spectral decomposition of $\bbH$, which reveals an appealing connection with matrix logarithmic mean \cite{BhatiaLogarithmic2008}. Specifically, letting $\bbH=\bbV\bbLambda\bbV^{-1}$, we get	
{\small{\begin{align*}
&\displaystyle\int_{0}^{\infty}\!\!\!\!(\bbH + \tau\bbI_{n})^{-1}\otimes(\bbH + \tau\bbI_{n})^{-1}\differential\tau\nonumber\\
&= \left(\bbV\!\otimes\!\bbV\right)\!\!\underbrace{\left(\displaystyle\int_{0}^{\infty}\!\!\!\!(\bbLambda + \tau\bbI_{n})^{-1}\!\otimes\!(\bbLambda + \tau\bbI_{n})^{-1}\!\differential\tau\right)}_{=:\bbL}\!\!\left(\bbV\!\otimes\!\bbV\right)^{-1},
\end{align*}}}
\!\!where the $n^{2}\times n^{2}$ diagonal matrix $\bbL$ has entries $\bbL_{jj}=\displaystyle\int_{0}^{\infty}\!\!\!\!\displaystyle\frac{\differential\tau}{(\lambda_{i} + \tau)^{2}} =\displaystyle\frac{1}{\lambda_{i}}$ for $j=(i-1)n+i$, $i=1,\hdots,n$; and $\bbL_{jj}=\displaystyle\int_{0}^{\infty}\!\!\!\!\displaystyle\frac{\differential\tau}{(\lambda_{i} + \tau)(\lambda_{j} + \tau)} = \displaystyle\frac{\log \lambda_{i} - \log \lambda_{j}}{\lambda_{i} - \lambda_{j}}$ for $j = [(i-1)n+i+1, i(n+1)]$, $i=1,\hdots,n-1$. Noting that for $a,b>0$, $\displaystyle\int_{0}^{1} a^{\tau}b^{1-\tau}\differential\tau = \displaystyle\frac{a-b}{\log a - \log b}$, we can rewrite (\ref{vecYFirstStep}) as
\begin{align}
&\left(\bbV\!\otimes\!\bbV\right)\!\!\underbrace{\left(\displaystyle\int_{0}^{1}\!\!\!\!\bbLambda^{\tau}\!\otimes\!\bbLambda^{1-\tau}\!\differential\tau\right)}_{=\bbL^{-1}}\!\!\left(\bbV\!\otimes\!\bbV\right)^{-1}\vec\left(\bbP_{0}^{1/2}\bbY\bbP_{0}^{1/2}\right) \nonumber\\
&= \left(\displaystyle\int_{0}^{1}\!\!\!\!\bbH^{\tau}\!\otimes\!\bbH^{1-\tau}\!\differential\tau\right)	\vec\left(\bbP_{0}^{1/2}\bbY\bbP_{0}^{1/2}\right) =  \vec(\log\bbH),\nonumber\\
&\stackrel{\footnotesize{(\ref{veckronprodidentity})}}{\Rightarrow} \displaystyle\int_{0}^{1}\bbH^{\tau}\bbP_{0}^{1/2}\bbY\bbP_{0}^{1/2}\bbH^{1-\tau}\differential\tau = \log\bbH,
\label{LogarithmicMeanMatrixEqn}
\end{align}
providing an integral characterization of (\ref{FinalY}), i.e., the gradient in Theorem \ref{PartialofFisherRaoGeodesicLength} satisfies the matrix equation (\ref{LogarithmicMeanMatrixEqn}). Indeed, the solution of (\ref{LogarithmicMeanMatrixEqn}) is given by \cite[equations (5.4.3) and (5.4.8)]{Hiai2010} 
{\small{\begin{align*}
\bbP_{0}^{1/2}\bbY\bbP_{0}^{1/2} &= \displaystyle\int_{0}^{\infty}\!\!\!\!(\bbH + \tau\bbI_{n})^{-1}\left(\log\bbH\right)(\bbH + \tau\bbI_{n})^{-1}\differential\tau \nonumber \\
&\stackrel{{\footnotesize{(\text{Lemma} \:\ref{MatrixIntegralOfInverseLoginverse})}}}{=} \bbH^{-1}\log\bbH =\allowbreak (\log\bbH)\bbH^{-1},
\end{align*}}}
thus establishing the equivalence between (\ref{FinalY}) and (\ref{LogarithmicMeanMatrixEqn}).
\end{remark}


\end{document}